\documentclass{article}

\usepackage[english]{babel}
\usepackage[letterpaper,top=2cm,bottom=2cm,left=3cm,right=3cm,marginparwidth=1.75cm]{geometry}

\usepackage{blindtext}
\usepackage{titlesec}
\usepackage{amsmath}
\usepackage{amssymb}
\usepackage{amsfonts}
\usepackage{graphicx}
\usepackage[colorlinks=true, allcolors=blue]{hyperref}
\usepackage{bbm}
\usepackage{mathabx}
\usepackage{amsthm}
\usepackage{tikz} 
\usepackage{mathtools}
\usepackage{hyperref}
\usepackage[capitalise]{cleveref}
\usepackage{animate}
\usepackage{xcolor}
\usepackage{graphicx}
\usepackage{xparse}
\usepackage{enumitem}
\usepackage{titlesec}

\titleformat*{\section}{\LARGE\bfseries}
\titleformat*{\subsection}{\Large\bfseries}

\newtheorem{theorem}{Theorem}
\newtheorem*{theorem*}{Theorem}
\newtheorem{lemma}[theorem]{Lemma}

\theoremstyle{definition}
\newtheorem{definition}[theorem]{Definition}

\newtheorem*{remark*}{Remark}

\DeclareMathOperator{\Cov}{Cov}

\newcommand{\norm}[1]{\lVert #1 \rVert}

\ExplSyntaxOn
\NewDocumentCommand{\ctriangle}{O{1} m m m}{
  \tl_set:Nn \l_tmpa_tl { \tl_trim_spaces:n {#2} }
  \tl_set:Nn \l_tmpb_tl { \tl_trim_spaces:n {#3} }
  \tl_set:Nn \l_tmpc_tl { \tl_trim_spaces:n {#4} }
  \tikz[baseline=(A), scale=#1]{
    \coordinate (A) at (0,0);
    \coordinate (B) at (1cm,0);
    \coordinate (C) at (0.5cm,0.866cm);

    \draw[\l_tmpa_tl, thick] (A) -- (B);
    \draw[\l_tmpb_tl, thick] (B) -- (C);
    \draw[\l_tmpc_tl, thick] (C) -- (A);

    \fill (A) circle (7pt);
    \fill (B) circle (7pt);
    \fill (C) circle (7pt);
  }
}
\ExplSyntaxOff

\newcommand{\cedge}[2][1]{%
  \tikz[baseline=-0.5ex,scale=#1]{%
    \draw[#2, thick] (0,0) -- (0,1cm);
    \filldraw[fill=black, draw=#2, thick] (0,0) circle[radius=7pt];
    \filldraw[fill=black, draw=#2, thick] (0,1cm) circle[radius=7pt];
     }%
}

\title{Exact Realization of $\{\cedge[0.2]{black},\ctriangle[0.3]{white}{white}{white},\ctriangle[0.3]{black}{white}{white},\ctriangle[0.3]{white}{black}{black},\ctriangle[0.3]{black}{black}{black}\}$ Induced Subgraph Counts}
\author{Hillel Schein\thanks{Department of Mathematics,
Technion---Israel Institute of Technology, Haifa, Israel.\\
Email: \href{mailto:hillel.raz@campus.technion.ac.il}
{\texttt{hillel.raz@campus.technion.ac.il}}.}}
\date{November 2026}

\begin{document}

\maketitle
\begin{abstract}
The study of three-vertex subgraph densities has developed along two complementary lines: exact finite realization in large graphs of a given order, and asymptotic approximate realization by sequences of graphs. We bridge these viewpoints in the case of the graphs $\{\cedge[0.15]{black},\ctriangle[0.2]{white}{white}{white},\ctriangle[0.2]{black}{white}{white},\ctriangle[0.2]{white}{black}{black},\ctriangle[0.2]{black}{black}{black}\}$. By analyzing local modifications of finite samples from graphons, we show that, for any nonempty subset of these densities, every asymptotically feasible target vector lying a fixed positive distance from the boundary of the corresponding feasible region is realized exactly by a graph of every sufficiently large order, subject only to the natural arithmetic constraints. The result holds for every nonempty projection onto any subset of induced three-vertex densities and the edge density. Thus, in the interior of these feasible regions, asymptotic feasibility implies exact finite realizability up to the unavoidable arithmetic obstructions.
\end{abstract}



\section{Introduction}

Given finite graphs $H_{1},\dots,H_{k}$, which prescribed counts of these graphs can occur simultaneously in a large graph? Research on this topic has developed in two seemingly separate directions. One direction studies whether the corresponding densities can be approached by
graphs whose orders tend to infinity (see~\cite{razborov2008minimal,huang20143,glebov2016densities}). The other asks whether large graphs of a specified order have precisely the prescribed counts (see~\cite{janson1992probabilistic,sah2024local,kittipassorn2025many}).

In this work, we bridge these two directions for the edge and the four three-vertex graphs $\{\cedge[0.15]{black},\ctriangle[0.2]{white}{white}{white},\ctriangle[0.2]{black}{white}{white},\ctriangle[0.2]{white}{black}{black},\ctriangle[0.2]{black}{black}{black}\}$. We prove that, for induced counts of these graphs, explicit arithmetic conditions
characterize which prescribed counts are realized exactly in sufficiently large graphs throughout the interior of every corresponding feasible region. The result allows any subset of
the counts to be prescribed. For each fixed distance from the boundary, the threshold on the number of vertices is uniform over all target density vectors. Thus, a known asymptotic feasible region determines exactly which
admissible density vectors in its interior are realized at every sufficiently large order.

For finite simple graphs $H,G$, write $v(H)=|V(H)|$ and define the induced count of $H$ in $G$ by
$$\#H(G)=\left|
\left\{
S\subseteq V(G): |S|=v(H),\ G[S]\simeq H
\right\}
\right|,$$
where $G[S]$ is the subgraph induced by $S$. The induced density of $H$ in $G$ is $\binom{v(G)}{v(H)}^{-1}\#H(G)$.

Let $\mathcal{H}$ be an ordered selection of $\{\cedge[0.15]{black},\ctriangle[0.2]{white}{white}{white},\ctriangle[0.2]{black}{white}{white},\ctriangle[0.2]{white}{black}{black},\ctriangle[0.2]{black}{black}{black}\}$. Given $\delta>0$, we say that $v\in (0,1)^{|\mathcal{H}|}$ is $\delta,\mathcal{H}$-interior if
for every $u$ in the Euclidean ball $B(v,\delta)$ there is a sequence of finite graphs $(G_n)$ such that
$$\left(\frac{\#H(G_{n})}{\binom{v(G_{n})}{v(H)}}\right)_{H\in\mathcal{H}}\rightarrow u,\qquad v(G_{n})\rightarrow\infty.$$

\begin{theorem}[Exact realization]\label{thm:main}
Let $\delta>0$, and let $\mathcal{H}\subset\{\cedge[0.15]{black},\ctriangle[0.2]{white}{white}{white},\ctriangle[0.2]{black}{white}{white},\ctriangle[0.2]{white}{black}{black},\ctriangle[0.2]{black}{black}{black}\}$ with $1\leq|\mathcal{H}|\leq 3$. There exists $N_{\delta,\mathcal{H}}$ such that for any $n>N_{\delta,\mathcal{H}}$ and any $\delta,\mathcal{H}$-interior vector $v$, the following are equivalent:
\begin{itemize}
     \item There exists a graph $G$ on $n$ vertices such that
    $$\#H(G)=\binom{n}{v(H)}v_{H},\qquad H\in\mathcal{H}.$$
    \item There exist nonnegative integers $N_{\cedge[0.1]{black}},N_{\ctriangle[0.15]{white}{white}{white}},N_{\ctriangle[0.15]{black}{white}{white}},N_{\ctriangle[0.15]{white}{black}{black}},N_{\ctriangle[0.15]{black}{black}{black}}$ satisfying
$$N_{H}=\binom{n}{v(H)}v_{H},\qquad H\in\mathcal{H},$$
$$N_{\ctriangle[0.15]{white}{white}{white}}+N_{\ctriangle[0.15]{black}{white}{white}}+N_{\ctriangle[0.15]{white}{black}{black}}+N_{\ctriangle[0.15]{black}{black}{black}}=\binom{n}{3},\qquad (n-2)N_{\cedge[0.1]{black}}=N_{\ctriangle[0.15]{black}{white}{white}}+2N_{\ctriangle[0.15]{white}{black}{black}}+3N_{\ctriangle[0.15]{black}{black}{black}}.$$
\end{itemize}

\end{theorem}

The two arithmetic conditions originate from identities of subgraph counts. Every three-element vertex
set induces exactly one of the four three-vertex graphs, giving the first identity. Counting incidences between edges and three-element vertex sets gives the second, since every edge lies in $n-2$ such sets. Any three of the five counts determine the other two, so three prescribed coordinates suffice to specify
the entire profile.

For fewer prescribed coordinates, interiority is required only in their projected feasible region. The
integers $N_{\cedge[0.1]{black}},N_{\ctriangle[0.15]{white}{white}{white}},N_{\ctriangle[0.15]{black}{white}{white}},N_{\ctriangle[0.15]{white}{black}{black}},N_{\ctriangle[0.15]{black}{black}{black}}$ are a completion of those coordinates, not further prescribed targets. The threshold $N_{\delta,\mathcal{H}}$
is common to all eligible vectors $v$; in particular, the targets may vary with $n$.
We also remark that tracking the constants in the proof gives a polynomial bound in $\delta^{-1}$, but we do not
record or optimize the resulting exponent.

We illustrate \cref{thm:main} by applying it to an existing asymptotic result. Consider the asymptotically feasible region of the $\cedge[0.15]{black}$ and $\ctriangle[0.2]{black}{black}{black}$ densities, shown in \cref{fig:edge-triangle-region}.
\begin{figure}[h]
\centering
\includegraphics[width=0.4\textwidth]{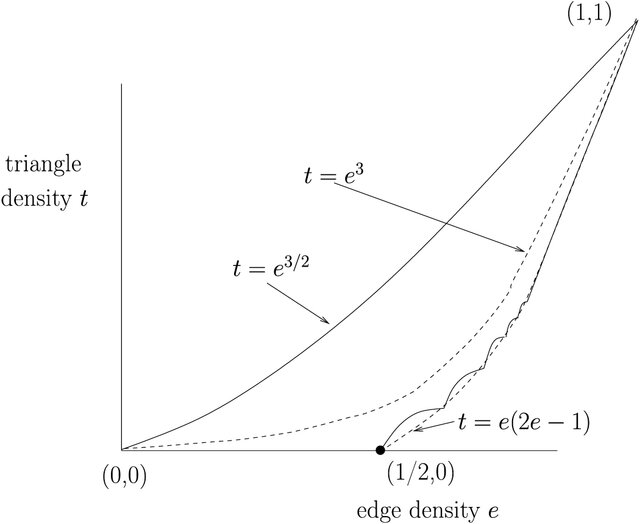}
\caption{Asymptotically feasible region of the $\cedge[0.2]{black}$ and $\ctriangle[0.3]{black}{black}{black}$ densities. Image from~\cite{radin2013phase}.}
\label{fig:edge-triangle-region}
\end{figure}

The upper bound is classical, and the lower bound was proved by Razborov in~\cite{razborov2008minimal}. All points in the interior of the region can be shown to be asymptotically feasible using graph limits (see~\cite{lovasz2012large}). In this case, the arithmetic conditions impose no constraint other than integrality. Thus, \cref{thm:main} implies that for any $\delta>0$ there exists $N_{\delta}$ such that, for any $n>N_{\delta}$ and any point $(a,b)$ in the region at distance at least $\delta$ from its boundary satisfying $\binom{n}{2}a,\binom{n}{3}b\in\mathbb{N}$, there exists a graph $G$ of order $n$ with $\cedge[0.15]{black}$ and $\ctriangle[0.2]{black}{black}{black}$ densities precisely $(a,b)$.

Similar applications give exact realizability for prescribed interior densities of any such subset $\mathcal{H}\subset\{\cedge[0.15]{black},\ctriangle[0.2]{white}{white}{white},\ctriangle[0.2]{black}{white}{white},\ctriangle[0.2]{white}{black}{black},\ctriangle[0.2]{black}{black}{black}\}$. The asymptotically feasible regions for the pairs $(\ctriangle[0.2]{white}{white}{white},\ctriangle[0.2]{black}{white}{white}),(\ctriangle[0.2]{white}{white}{white},\ctriangle[0.2]{black}{black}{black}),(\ctriangle[0.2]{black}{white}{white},\ctriangle[0.2]{white}{black}{black}),(\ctriangle[0.2]{white}{black}{black},\ctriangle[0.2]{black}{black}{black})$ were characterized in~\cite{huang20143,glebov2016densities}; applying \cref{thm:main} gives the corresponding exact realization results in their interiors.

We prove \cref{thm:main} using graphons, the graph limits used to study asymptotic feasibility of these densities (see~\cite{lovasz2006limits}). Graphons define a family of random graph models whose subgraph densities approximate all feasible density vectors. We prove that, with positive probability, a graph sampled from a suitable graphon can be modified to have the prescribed densities. This proof method generalizes the algorithmic probabilistic construction in~\cite{janson1992probabilistic}, which produces finite graphs whose three-vertex densities equal the corresponding expectations in the Erdős–Rényi random graph model $G(n,p)$. In this model, each edge is present independently with probability $p$.

However, unlike $G(n,p)$, the random graph model associated with an arbitrary graphon need not allow controlled modifications of subgraph densities with positive probability. To resolve this, we prove that every interior target density vector has a graphon representative with the required properties. To analyze how modifications affect the chosen subgraph densities, we apply results from the theory of generalized $U$-statistics~\cite{janson1991asymptotic}, decomposing each density and making specialized modifications at different scales.

We introduce the tools from graph limits and generalized $U$-statistics in \cref{sec:prelim}, prove the existence of a suitable random model in \cref{sec:robust-representative}, and present the algorithmic modifications and prove their availability with positive probability in \cref{sec:first,sec:paired,sec:coarse,sec:precise}. Finally, in \cref{sec:proof}, a union bound shows that all required modifications are simultaneously available with positive probability, yielding a graph with the prescribed densities.

\section{Preliminaries}\label{sec:prelim}
\subsection{Graphs and Graphons}
Throughout the paper, graphs are simple and undirected: $G=(V,E)$, with $E\subseteq\binom{V}{2}$. We write $V(G)$ and $E(G)$ for the vertex and edge sets of $G$, and set $v(G)=|V(G)|$ and $e(G)=|E(G)|$. We sometimes omit the underlying graph $G$ when it is clear. As in the introduction, for finite graphs $H,G$ we define the induced count of $H$ in $G$ by
$$\#H(G)=\left|
\left\{
S\subseteq V(G): |S|=v(H),\ G[S]\simeq H
\right\}
\right|.$$

We now introduce graphons (see~\cite{lovasz2006limits} for an overview). A graphon is a symmetric measurable function $W:[0,1]^2\to[0,1]$ that represents a limit of dense graphs through their subgraph densities. For a graphon $W$ and a finite graph $H$, put $k=v(H)$ and label its vertices by $[k]=\{1,\ldots,k\}$. We define the density $t(H,W)$ and the labelled induced density $t_{\mathrm{ind}}(H,W)$ by
\[
t(H,W)
=
\int_{[0,1]^k}
\prod_{ij\in E(H)}W(x_i,x_j)\,d\mathbf x,
\]
\[
t_{\mathrm{ind}}(H,W)
=
\int_{[0,1]^k}
\prod_{ij\in E(H)}W(x_i,x_j)
\prod_{ij\in\binom{[k]}2\setminus E(H)}
\bigl(1-W(x_i,x_j)\bigr)\,d\mathbf x.
\]

The asymptotically feasible density vectors of finite graphs are precisely those represented by graphons~\cite{lovasz2006limits}, with the conversion between labelled and unlabelled induced densities specified below. Inclusion--exclusion expresses induced densities as linear combinations of ordinary densities, and conversely. For three-vertex graphs, we have
$$t_{\mathrm{ind}}(\ctriangle[0.3]{black}{black}{black},W)=t(\ctriangle[0.3]{black}{black}{black},W),\quad t_{\mathrm{ind}}(\ctriangle[0.3]{white}{black}{black},W)=t(\ctriangle[0.3]{white}{black}{black},W)-t(\ctriangle[0.3]{black}{black}{black},W),$$
$$t_{\mathrm{ind}}(\ctriangle[0.3]{black}{white}{white},W)=t(\ctriangle[0.3]{black}{white}{white},W)-2t(\ctriangle[0.3]{white}{black}{black},W)+t(\ctriangle[0.3]{black}{black}{black},W),\quad t_{\mathrm{ind}}(\ctriangle[0.3]{white}{white}{white},W)=1-3t(\ctriangle[0.3]{black}{white}{white},W)+3t(\ctriangle[0.3]{white}{black}{black},W)-t(\ctriangle[0.3]{black}{black}{black},W).$$
This transformation is triangular and invertible. Thus the four induced three-vertex densities, which also determine the edge density, carry the same information as the corresponding ordinary densities. Since $t(\ctriangle[0.3]{white}{white}{white},W)\equiv 1$ and $t(\ctriangle[0.3]{black}{white}{white},W)=t(\cedge[0.15]{black},W)$, it suffices to consider the edge, $2$-star and triangle densities. We use these coordinates in \cref{sec:robust-representative,sec:first}.

These integrals have a probabilistic interpretation. Sample $v(H)$ independent uniform points in $[0,1]$ and, conditional on these points, include each edge $ij$ independently with probability $W(x_i,x_j)$. Then $t(H,W)$ is the probability that all edges of a fixed labelled copy of $H$ are present, and $t_{\mathrm{ind}}(H,W)$ is the probability that the sampled graph is exactly that copy. The probability that the sampled graph is isomorphic to $H$ is $\frac{v(H)!}{|\operatorname{Aut}(H)|}t_{\mathrm{ind}}(H,W)$, where $\operatorname{Aut}(H)$ is the automorphism group of $H$. This factor is needed because $\#H(G)$ counts vertex sets rather than labellings.

The same construction gives a distribution on graphs of order $n$ for every $n\in\mathbb{N}$ and graphon $W$.

\begin{definition}
    \label{def:W-random-graph}
    Let $W:[0,1]^{2}\rightarrow[0,1]$ be a graphon, and let $n\in\mathbb{N}$. Sample mutually independent random variables $(X_{i})_{i\in[n]}$ and $(Y_{ij})_{1\leq i<j\leq n}$, all uniformly distributed on $[0,1]$. Define the $W$-random graph $W_{n}$ on $[n]$ by
    $$ij\in E(W_{n})\iff Y_{ij}\leq W(X_{i},X_{j}),\qquad 1\leq i<j\leq n.$$
\end{definition}
This standard sampling model has been studied extensively (see~\cite{lovasz2006limits,bhattacharya2023fluctuations,zhang2022berry}). The following identity relates expected induced counts in the sample to graphon densities.
\begin{lemma}
    Let $W:[0,1]^{2}\rightarrow[0,1]$ be a graphon, let $H$ be a finite graph with $k=v(H)$, and let $n\in\mathbb{N}$. Then
    \[
\mathbb E\bigl[\#H(W_n)\bigr]
=
\binom nk\frac{k!}{|\operatorname{Aut}(H)|}\,
t_{\mathrm{ind}}(H,W).
\]
\end{lemma}
\begin{proof}
Each $k$-element vertex set has probability
\[
\frac{k!}{|\operatorname{Aut}(H)|}\,t_{\mathrm{ind}}(H,W)
\]
of inducing a graph isomorphic to $H$. Summing over these sets
proves the identity.
\end{proof}

\subsection{Generalized $U$-statistics}

Generalized $U$-statistics include subgraph counts and densities of graphon samples. They have been studied both in this setting and more generally (see~\cite{janson1991asymptotic,zhang2022berry,bhattacharya2023fluctuations}). We first give their definition.
\begin{definition}\label{def:generalized-u-statistics}
Let $n\in\mathbb{N}$, and let $(X_i)_{i\in[n]}$ and $(Y_{ij})_{1\leq i<j\leq n}$ be two independent families of real-valued random variables, each consisting of i.i.d. variables. Set
    $Y_{ji}=Y_{ij}$. Let $k\in\mathbb{N}$ satisfy $k\leq n$, and let $h:\mathbb{R}^{k+\binom{k}{2}}\rightarrow\mathbb{R}$ be measurable. Define the generalized $U$-statistic $S_{n}(h)$ and its normalization $U_{n}(h)$ by
\[
S_n(h)
=
\sum_{\varphi:[k]\hookrightarrow[n]}h(\varphi),\quad U_{n}(h)=\frac{1}{(n)_{k}}S_{n}(h),\quad h(\varphi)=
h\bigl(
(X_{\varphi(i)})_{i\in[k]},
(Y_{\varphi(i)\varphi(j)})_{i<j}
\bigr),
\]
where $(n)_{k}=\prod_{i=0}^{k-1}(n-i)$, with $(n)_0=1$.
\end{definition}

In the $W$-random graph model $W_{n}$, let $h$ be the indicator that the sampled labelled graph on $[k]$ is a fixed graph $H$. Then
\[
S_n(h)=|\operatorname{Aut}(H)|\,\#H(W_n),\qquad
\mathbb E[U_n(h)]=t_{\mathrm{ind}}(H,W).
\]
If $h$ only requires the edges of $H$ to be present, then $U_n(h)$ is the normalized injective density $t(H,W_n)$ defined at the start of \cref{sec:robust-representative}.

Generalized $U$-statistics admit an orthogonal decomposition whose components are indexed by graphs~\cite{janson1991asymptotic}. We give a proof and record explicitly the components on at most two vertices.

\begin{lemma}[Orthogonal decomposition]
\label{lem:orthogonal-decomposition}
Let $n,k\in\mathbb{N}$ with $k\leq n$, and let $h:\mathbb{R}^{k+\binom{k}{2}}\rightarrow\mathbb{R}$ be a bounded measurable kernel on $k$ vertices, in the setting of \cref{def:generalized-u-statistics}. There exist
bounded symmetric component kernels $h_H$, indexed by all graphs $H\subseteq K_{k}$ up to isomorphism, such that
\[
S_n(h)
=
\sum_{H\subseteq K_{k}}
(n-v(H))_{k-v(H)}
\sum_{\substack{S\subseteq[n]\\|S|=v(H)}}h_H(S).
\]
The outer sum contains one representative of each isomorphism type. Subgraphs need not be induced and may have isolated vertices; $\emptyset$ denotes the graph with no vertices.
Here $h_H(S)$ is evaluated on the vertex and pair variables indexed
by $S$; symmetry makes the evaluation independent of their
ordering.

For $T\subseteq[n]$, let $\mathcal F_T
=
\sigma\bigl(X_i:i\in T;\ Y_{ij}:i,j\in T,\ i<j\bigr)$.

Every nonempty component has mean zero and satisfies
$S\not\subseteq T
\quad\Longrightarrow\quad
\mathbb E[h_H(S)\mid\mathcal F_T]=0$.

Terms corresponding to distinct pairs $(H,S)$ are uncorrelated.
Consequently, for every $\emptyset\neq H$ and every deterministic family
$\mathcal A\subseteq\binom{[n]}{v(H)}$,
\[
\operatorname{Var}\left(\sum_{S\in\mathcal A}h_H(S)\right)
=
|\mathcal A|\,\mathbb E[h_H^2],
\]

Writing $\mu=\mathbb Eh$, the components on at most two vertices have the following explicit forms. The sums over $i\neq j$ are over ordered pairs.
\begin{itemize}
    \item $h_{\emptyset}=\mu$.
    \item $h_{\bullet}(x)
=
\sum_{i=1}^k
\bigl(\mathbb E[h\mid X_i=x]-\mu\bigr)$.
\item $h_{\bullet\hspace{1mm}\bullet}(x,x')
=
\sum_{i\neq j\in[k]}
\left(
\mathbb E[h\mid X_i=x,X_j=x']
-\mathbb E[h\mid X_i=x]
-\mathbb E[h\mid X_j=x']
+\mu
\right)$.
\item $h_{\cedge[0.1]{black}}(x,x',y)
=
\sum_{i\neq j\in[k]}
\left(
\mathbb E[h\mid X_i=x,X_j=x',Y_{ij}=y]
-\mathbb E[h\mid X_i=x,X_j=x']
\right)$.

\end{itemize}
\end{lemma}

\begin{proof}
For a labelled subgraph $A\subseteq K_k$, let $\mathcal F_A
=
\sigma\bigl(X_i:i\in V(A);\ Y_{ij}:ij\in E(A)\bigr)$.

Define the labelled projections recursively by
$$q_\varnothing=\mathbb Eh,
\qquad
q_A=\mathbb E[h\mid\mathcal F_A]-\sum_{B\subsetneq A}q_B
\qquad(A\neq\varnothing).$$
Thus
$$\mathbb E[h\mid\mathcal F_A]=\sum_{B\subseteq A}q_B,
\qquad
h=\sum_{A\subseteq K_k}q_A.$$

For each isomorphism type $H\subseteq K_k$, write $r=v(H)$ and set
$$h_H
=
\sum_{\substack{A\subseteq K_k\\A\simeq H}}
\ \sum_{\psi:V(A)\hookrightarrow[r]}q_A(\psi).$$
The inner sum runs over all bijections from $V(A)$ to $[r]$, and $q_A(\psi)$ denotes $q_A$ with its vertex and pair arguments relabelled by $\psi$. Summing over these bijections makes $h_H$ symmetric. The recursive construction bounds each component kernel by a constant depending only on $k$ times $\lVert h\rVert_\infty$. Each injection of $V(A)$ into $[n]$
has $(n-r)_{k-r}$ extensions to an injection of $[k]$ into $[n]$. Summing the
labelled decomposition over these injections and grouping by
isomorphism type and image vertex set gives the asserted formula
for $S_n(h)$.

To prove degeneracy, independence of the underlying variables gives
\[
\mathbb E[q_A\mid\mathcal F_B]
=
\mathbb E[q_A\mid\mathcal F_{A\cap B}].
\]
We show inductively that this is zero whenever $A\not\subseteq B$.
For $J\subsetneq A$, the recursion and the induction hypothesis
yield
\[
\begin{aligned}
\mathbb E[q_A\mid\mathcal F_J]
&=
\mathbb E[h\mid\mathcal F_J]
-
\sum_{B\subsetneq A}\mathbb E[q_B\mid\mathcal F_J]\\
&=
\mathbb E[h\mid\mathcal F_J]
-
\sum_{B\subseteq J}q_B
=0.
\end{aligned}
\]
Indeed, terms with $B\not\subseteq J$ vanish after conditioning on
$\mathcal F_{B\cap J}$. This proves degeneracy and centering for all
nonempty labelled projections, and hence the stated conditional
degeneracy for $h_H(S)$.

Two embedded labelled projections with distinct image graphs are
uncorrelated: one image is not contained in the other, and
conditioning on the latter makes the former term vanish. Distinct
pairs $(H,S)$ contain no common image graph, so their component
terms are also uncorrelated. Expanding the variance therefore
proves the identity for $\mathcal A$.
Finally, applying the recursion to graphs on one or two vertices
gives the explicit component formulas.
\end{proof}

For fixed $k$, the variance identity shows that the contribution indexed by a nonempty graph $H$ to $S_n(h)$ has variance $O(n^{2k-v(H)})$, with constants depending only on $k$ and the bound on $h$. In particular, every induced three-vertex count has variance $O(n^5)$, as used in \cref{sec:first}. Components indexed by edgeless graphs depend only on the vertex variables, so changing only the pair variables leaves them unchanged. Among components on at most two vertices, only the edge component can change; its explicit formula is used in \cref{sec:robust-representative}.

\section{Choosing a robust representative}
\label{sec:robust-representative}

For finite graphs $H,G$, let $t(H,G)$ be the number of injective edge-preserving maps from $V(H)$ to $V(G)$, divided by $(v(G))_{v(H)}$. Thus $\mathbbm{E}[t(H,W_n)]=t(H,W)$, with the normalization of $U_n$ in \cref{def:generalized-u-statistics}.

We will modify a sample $W_n$ by changing edge variables while keeping the vertex variables fixed. We therefore need a representative that permits both edge additions and removals and supplies a nondegenerate response in every required direction. By \cref{sec:prelim}, it suffices to work with the edge, $2$-star and triangle densities. Write
$$d_W(x)=\int_0^1 W(x,z)\,dz,\qquad c_W(x,y)=\int_0^1 W(x,z)W(y,z)\,dz$$
for the degree and common-neighbour functions. The formula in \cref{lem:orthogonal-decomposition} shows that adding an edge $ij$ changes the edge-component kernels by twice the vector
\[
D_W(X_i,X_j)\coloneqq\bigl(1,d_W(X_i)+d_W(X_j),3c_W(X_i,X_j)\bigr).
\]
The corresponding contribution to the normalized density vector changes by $\binom n2^{-1}D_W(X_i,X_j)$; removing the edge gives the opposite change.

We seek a representative $W$ with two properties:
\begin{itemize}
\item $\eta_\delta<W(x,y)<1-\eta_\delta$ for all $x,y\in[0,1]$, with $\eta_\delta>0$ depending only on $\delta$;
\item $D_W$ is quantitatively nondegenerate in every direction.
\end{itemize}
The first gives each edge state conditional probability at least $\eta_\delta$; the second supplies directional increments for the correction in \cref{sec:first}. We first impose both properties by a transformation $F_\eta$, then choose its input so that the output has the prescribed densities.

\begin{lemma}[The representative transformation]
\label{lem:s3-transformation}
Let $0<\eta\leq1/12$. For a graphon $U$, define
\[
I_\eta U=(1-6\eta)U+3\eta.
\]
For a graphon $V$ with $3\eta\leq V\leq1-3\eta$, define
\[
S_\eta V(x,y)=
\begin{cases}
V(2x,2y)+2\eta,&x,y<1/2,\\
V(2x,2y-1),&x<1/2\leq y,\\
V(2x-1,2y),&y<1/2\leq x,\\
V(2x-1,2y-1)-2\eta,&x,y\geq1/2.
\end{cases}
\]
Set $F_\eta=S_\eta\circ I_\eta$. Then $F_\eta$ is $L^1$-continuous, and $W=F_\eta U$ satisfies $\eta\leq W\leq1-\eta$. The edge, $2$-star and triangle densities of $W$ differ from those of $U$ by $O(\eta)$, uniformly in $U$. Moreover,
$$\inf_{\substack{u\in\mathbb R^3\\\lVert u\rVert=1}}\int_{[0,1]^2}|u\cdot D_W(x,y)|\,dx\,dy=\Omega(\eta^2).$$
\end{lemma}

\begin{proof}
The first transformation leaves a margin of $3\eta$; the offsets of size at most $2\eta$ in the second leave a margin of $\eta$. Copying into the four half-size squares preserves all densities, since a uniform point in either half rescales to a uniform point in $[0,1]$. Each perturbation has pointwise size $O(\eta)$, so the defining integrals give the stated $O(\eta)$ density changes. Also,
\[
\|F_\eta U-F_\eta U'\|_1=(1-6\eta)\|U-U'\|_1,
\]
which proves continuity.

To prove nondegeneracy, separate the rescaled position from the half interval containing it. Define
\[
\tau(x)=
\begin{cases}2x,&x<1/2,\\2x-1,&x\geq1/2,\end{cases}
\qquad
\sigma(x)=
\begin{cases}1,&x<1/2,\\-1,&x\geq1/2.\end{cases}
\]
For uniform $x\in[0,1]$, $\tau(x)$ is uniform and $\sigma(x)$ is an independent uniform sign. For $W=S_\eta V$, direct calculation gives
\[
\begin{aligned}
d_W(x)&=d_V(\tau(x))+\eta\sigma(x),\\
c_W(x,y)&=c_V(\tau(x),\tau(y))
 +\eta\sigma(x)d_V(\tau(y))+\eta\sigma(y)d_V(\tau(x))
 +\eta^2\bigl(1+\sigma(x)\sigma(y)\bigr).
\end{aligned}
\]
The linear and product terms in the signs separate the three response directions. Let $r$ be the $L^1$ norm of $a+b(d_W(x)+d_W(y))+c\,c_W(x,y)$. Averaging against $\sigma(x)\sigma(y)$ and against $\sigma(x)$, with $\tau(x),\tau(y)$ fixed, isolates $c\eta^2$ and $\eta(b+c\,d_V(\tau(y)))$, respectively. Taking absolute values and integrating gives
\[
|c|\eta^2\leq r,
\qquad
\eta\int_0^1|b+c\,d_V(y)|\,dy\leq r.
\]
Since $0\leq d_V,c_V\leq1$, these inequalities and the mean of the original expression give $|b|\leq2r/\eta^2$ and $|a|\leq6r/\eta^2$. Hence $r\geq\eta^2\|(a,b,c)\|_2/9$. Taking $(a,b,c)=(u_1,u_2,3u_3)$ yields $\int_{[0,1]^2}|u\cdot D_W(x,y)|\,dx\,dy\geq\eta^2/9$ for every unit vector $u$.
\end{proof}

The transformation gives every input the required properties, but perturbs its densities. We now use the $\delta$-neighbourhood of the target to choose an input whose image has exactly the prescribed densities.

\begin{lemma}[Interior representatives]
\label{lem:s3-representatives}
Let $\mathcal H$ be a nonempty selection of at most three graphs from
$\{\cedge[0.15]{black},\ctriangle[0.2]{white}{white}{white},\ctriangle[0.2]{black}{white}{white},\ctriangle[0.2]{white}{black}{black},\ctriangle[0.2]{black}{black}{black}\}$, and let $v$ be $\delta,\mathcal H$-interior. There exist $\eta=\Theta(\delta)$ and a graphon $U$ such that $W=F_\eta U$ satisfies
\[
\frac{v(H)!}{|\operatorname{Aut}(H)|}\,t_{\mathrm{ind}}(H,W)=v_H
\qquad(H\in\mathcal H).
\]
\end{lemma}

\begin{proof}

For each choice of $\mathcal{H}$, choose $\eta$ to be a sufficiently small positive multiple of $\delta$. One prescribed density is handled by interpolation, and three by inverting the induced map on densities. For two densities, we first treat the edge coordinate directly, then use boundary loops for four induced-density pairs and a level-set path for the remaining two complementary cases.

\smallskip\noindent\emph{One prescribed density.}
Choose graphons $U_-,U_+$ representing $v-\delta/2$ and $v+\delta/2$, and set $U_s=(1-s)U_-+sU_+$ for $s\in[0,1]$. The prescribed density of $F_\eta U_s$ varies continuously. The $O(\eta)$ perturbation in \cref{lem:s3-transformation} leaves its endpoint values on opposite sides of $v$, so the intermediate value theorem gives the required graphon.

\smallskip\noindent\emph{Three prescribed densities.}
Here we invert the affine map that $F_\eta$ induces on the full density profile. Set $a=1-6\eta$ and $b=3\eta$. Direct calculation gives
\begin{equation}
\begin{aligned}
t(\cedge[0.2]{black},F_\eta U)&=a\,t(\cedge[0.2]{black},U)+b,\\
t(\ctriangle[0.3]{white}{black}{black},F_\eta U)&=a^2t(\ctriangle[0.3]{white}{black}{black},U)+2ab\,t(\cedge[0.2]{black},U)+b^2+\eta^2,\\
t(\ctriangle[0.3]{black}{black}{black},F_\eta U)&=a^3t(\ctriangle[0.3]{black}{black}{black},U)+3a^2b\,t(\ctriangle[0.3]{white}{black}{black},U)
 +3a(b^2+\eta^2)t(\cedge[0.2]{black},U)+b^3+3b\eta^2.
\end{aligned}
\label{eq:s3-density-map}
\end{equation}
This map is triangular with nonzero diagonal coefficients $a,a^2,a^3$. The linear parts of this map and its inverse differ from the identity by $O(\eta)$, and their constant terms are $O(\eta)$. Any three prescribed coordinates determine the edge, $2$-star and triangle densities by a fixed invertible affine map. Thus $\delta,\mathcal H$-interiority gives a feasible neighbourhood of radius $\Theta(\delta)$ around the corresponding density vector. For sufficiently small $\eta=\Theta(\delta)$, its inverse image remains feasible. Choose a graphon $U$ realizing that inverse image; then $F_\eta U$ has the prescribed densities.

\smallskip\noindent\emph{Two prescribed densities, one of them the edge density.}
Let $H$ be the other graph. We compensate for the edge-density shift first, then interpolate in the $H$-density. Choose graphons $U_-,U_+$ with edge density $\frac{v_{\cedge[0.1]{black}}-3\eta}{1-6\eta}$ and prescribed $H$-densities $v_H-\delta/2$ and $v_H+\delta/2$. These pairs are feasible because the edge-coordinate shift is $O(\eta)$. Every convex combination of $U_-,U_+$ has the same edge density, so its image under $F_\eta$ has edge density $v_{\cedge[0.1]{black}}$. The $O(\eta)$ perturbation leaves the endpoint $H$-densities on opposite sides of $v_H$, and the intermediate value theorem gives both targets.

\smallskip\noindent\emph{Two induced densities: boundary paths.}

For these cases, it suffices to find an $L^1$-continuous graphon loop $U_\varphi$ whose density curve winds nontrivially around $v$ and stays at distance $\Omega(\delta)$ from it. Applying $F_\eta$ moves the curve by $O(\eta)$ and therefore preserves its winding number. Fill the graphon loop by
\[
(1-r)/2+rU_\varphi,\qquad 0\leq r\leq1,
\]
and apply $F_\eta$ throughout. At $r=0$ the graphon is independent of $\varphi$, so the transformed family is continuously parameterized by a disk. Its density image must contain $v$; otherwise radial contraction would contract a loop of nonzero winding number in $\mathbb R^2\setminus\{v\}$.

We obtain the loops from~\cite[Theorem~1.2]{huang20143} and~\cite[Theorems~13--14]{glebov2016densities}. All pairs below use the normalization in the lemma. The pair $(\ctriangle[0.2]{black}{white}{white},\ctriangle[0.2]{black}{black}{black})$ and its complement are treated separately afterward.

For $(\ctriangle[0.2]{black}{white}{white},\ctriangle[0.2]{white}{black}{black})$, the feasible region is the triangle with vertices $(0,0),(3/4,0),(0,3/4)$. Concatenate three families: two disjoint cliques of measure $r$ each, with an isolated remainder, as $r$ increases from $0$ to $1/2$; two equal blocks with within-block probability $r$ and cross probability $1-r$, as $r$ decreases from $1$ to $0$; and complete bipartite graphons with part measures $r,1-r$, as $r$ decreases from $1/2$ to $0$. Their density pairs are
\[
\bigl(6r^2(1-r),0\bigr),\qquad
\left(\frac38-\frac38(1-2r)^3,\frac38+\frac38(1-2r)^3\right),\qquad
\bigl(0,3r(1-r)\bigr).
\]
Using consecutive intervals for the blocks makes the graphon endpoints agree, with the empty graphon at both ends. This $L^1$-continuous loop traverses the boundary once.

For $(\ctriangle[0.2]{white}{black}{black},\ctriangle[0.2]{black}{black}{black})$, the complete multipartite boundary family in~\cite[Theorem~14 and its proof]{glebov2016densities} joins the balanced complete bipartite graphon to the complete graphon. Close it by decreasing a clique with isolates from measure $1$ to $0$, then increasing the cross probability of a balanced bipartite graphon from $0$ to $1$. These $L^1$-continuous paths have matching graphon endpoints and trace the boundary once. Complementation gives the loop for $(\ctriangle[0.2]{white}{white}{white},\ctriangle[0.2]{black}{white}{white})$. For these pairs and the preceding pair, interiority gives distance at least $\delta$ from the density loop.

For $(\ctriangle[0.2]{white}{white}{white},\ctriangle[0.2]{black}{black}{black})$, write $v=(x,y)$. The two upper arcs in~\cite[Theorem~1.2 and its proof]{huang20143} are the clique-with-isolates curve $((1-r)^2(1+2r),r^3)$, $0\leq r\leq1$, and its reflection across the diagonal. Closing each against the common lower boundary path gives two $L^1$-continuous graphon loops. Let $z_*=(a,a)$ be the non-endpoint intersection of the upper arcs, which lies on the outer boundary. If $x\geq y$, choose the clique loop: the feasible-region description places $v$ strictly inside this loop, and its only points off the outer boundary have the form $z=(s,t)$ with $s\leq a\leq t$. Hence
\[
\|z-z_*\|\leq t-s\leq\sqrt2\,\|z-v\|,
\qquad
\delta\leq\|v-z_*\|\leq(1+\sqrt2)\|z-v\|.
\]
All other points of the loop lie on the outer boundary and are at distance at least $\delta$ from $v$. If $x<y$, use the reflected loop and interchange the coordinates in this argument. Thus in each case the chosen density curve winds around $v$ and stays at distance at least $\delta/(1+\sqrt2)$ from it.

\smallskip\noindent\emph{Two induced densities: $(\ctriangle[0.2]{black}{white}{white},\ctriangle[0.2]{black}{black}{black}),(\ctriangle[0.2]{white}{white}{white},\ctriangle[0.2]{white}{black}{black})$.}
We treat the first pair and obtain the second by complementation. We seek a path of transformed graphons on which one linear combination of the densities is fixed while the triangle density crosses its target. Define
\begin{equation}
q(U)=t(\cedge[0.2]{black},U)-2t(\ctriangle[0.3]{white}{black}{black},U)
=t_{\mathrm{ind}}(\ctriangle[0.3]{black}{white}{white},U)-t_{\mathrm{ind}}(\ctriangle[0.3]{black}{black}{black},U)
=\frac18-2\int_0^1\bigl(d_U(x)-\tfrac14\bigr)^2\,dx.
\label{eq:s3-q}
\end{equation}
This function is concave on pointwise convex segments, and $-1\leq q(U)\leq1/8$. Since $q(U)\geq-t(\cedge[0.15]{black},U)\geq-1$, equality $q(U)=-1$ forces $U=1$ almost everywhere.

First, any two graphons $U_0,U_1$ with $q(U_0)=q(U_1)=c>-1$ can be joined by an $L^1$-continuous path on that level. Their convex combinations $V_s=(1-s)U_0+sU_1$, $s\in[0,1]$, satisfy $q(V_s)\geq c$. To restore equality, adjoin universal vertices of a suitable measure. For $0\leq t<1$, set
\[
V_{s,t}(x,y)=
\begin{cases}
V_s\bigl(x/(1-t),y/(1-t)\bigr),&x,y<1-t,\\
1,&\text{otherwise},
\end{cases}
\qquad V_{s,1}=1.
\]
This family is jointly $L^1$-continuous: this holds for step graphons, and compression multiplies $L^1$ approximation errors by $(1-t)^2$. Differentiating the degree expression for $q$ gives
\[
\frac{\partial}{\partial t}q(V_{s,t})
=-6(1-t)\int_0^1
 \bigl[t+(1-t)d_{V_s}(x)\bigr]\bigl(1-d_{V_s}(x)\bigr)\,dx<0
\qquad(0<t<1),
\]
since $V_s$ is not complete. As $q(V_{s,0})\geq c>-1=q(V_{s,1})$, there is a unique $t_s\in[0,1)$ with $q(V_{s,t_s})=c$. This root varies continuously with $s$: every limit of roots is a root at the limiting $s$ and cannot equal $1$. Thus $V_{s,t_s}$ is the required path, with endpoints $U_0,U_1$ because $t_0=t_1=0$.

To use this path after applying $F_\eta$, first rescale the input so that the transformed $q$-value depends only on its original value. For $\eta<1/12$, $\frac{1-12\eta}{1-6\eta}U$ is still a graphon. Set $a=1-12\eta$ in this case. The degree formula in the proof of \cref{lem:s3-transformation} gives
\[
q\left(F_\eta\left(\frac{1-12\eta}{1-6\eta}U\right)\right)
=a^2q(U)+\frac{1-a^2}{8}-2\eta^2.
\]
To make this equal the prescribed linear combination, choose the input level
\[
c=
\frac{
 \frac13v_{\ctriangle[0.15]{black}{white}{white}}
 -v_{\ctriangle[0.15]{black}{black}{black}}
 -(1-a^2)/8+2\eta^2
}{a^2}.
\]
It differs by $O(\eta)$ from $\frac13v_{\ctriangle[0.15]{black}{white}{white}}-v_{\ctriangle[0.15]{black}{black}{black}}$. The two density vectors
\[
\left(
3c+3v_{\ctriangle[0.15]{black}{black}{black}}\pm\frac{3\delta}{10},
\quad
v_{\ctriangle[0.15]{black}{black}{black}}\pm\frac{\delta}{10}
\right)
\]
bracket the target triangle density while keeping $q=c$. For sufficiently small $\eta=\Theta(\delta)$, both lie in $B(v,\delta)$. Choose graphons $U_0,U_1$ realizing the minus and plus vectors. By \cref{eq:s3-q}, both have $q=c$; their different triangle densities imply $c>-1$. The preceding construction therefore joins them by a path $V_{s,t_s}$ on that level.

Apply $F_\eta\bigl(\frac{1-12\eta}{1-6\eta}V_{s,t_s}\bigr)$ along this path. Its $q$-value is constantly $\frac13v_{\ctriangle[0.15]{black}{white}{white}}-v_{\ctriangle[0.15]{black}{black}{black}}$. For any graphon $U$, this rescaled transformation differs pointwise from $U(\tau(x),\tau(y))$ by at most $11\eta$. Copying preserves densities, so the triangle density changes by at most $33\eta$. Choose $33\eta<\delta/10$; the transformed endpoints still bracket the target triangle density. The intermediate value theorem attains that density, and the fixed $q$-value gives the prescribed one-edge induced density. Complementation is compatible with $F_\eta$ after exchanging its two half intervals, which gives the complementary case.
\end{proof}

The modification arguments require integer targets for all five counts, close to their expectations in the chosen representative. The arithmetic hypothesis of \cref{thm:main} gives a completion, but not necessarily a nearby one. The next lemma supplies such a completion while preserving every prescribed count.

\begin{lemma}[Arithmetic completion]
\label{lem:s3-arithmetic}
Let $\mathcal H\subseteq\{\cedge[0.15]{black},\ctriangle[0.2]{white}{white}{white},\ctriangle[0.2]{black}{white}{white},\ctriangle[0.2]{white}{black}{black},\ctriangle[0.2]{black}{black}{black}\}$ with $1\leq|\mathcal H|\leq3$, let $n\in\mathbb N$ and $\eta\in(0,\frac14)$, and let $W$ be a graphon with $\eta<W<1-\eta$. Suppose that $(\mathbbm{E}[\#H(W_n)])_{H\in\mathcal H}$ admits an arithmetic completion as in \cref{thm:main}. If $n$ is sufficiently large, depending only on $\eta$, there are nonnegative integers $N_H'$, indexed by the same five graphs and satisfying the two count identities in \cref{thm:main}, such that

$$N_H'=\mathbbm{E}[\#H(W_n)],\qquad (H\in\mathcal H),$$
$$\forall H\in\{\ctriangle[0.2]{white}{white}{white},\ctriangle[0.2]{black}{white}{white},\ctriangle[0.2]{white}{black}{black},\ctriangle[0.2]{black}{black}{black}\},\hspace{2mm}
\bigl|N_{H}'-\mathbbm{E}[\#H(W_n)]\bigr|\leq2n,\qquad \bigl|N_{\cedge[0.15]{black}}'-\mathbbm{E}[\#\cedge[0.15]{black}(W_n)]\bigr|\leq24.$$

\end{lemma}

\begin{proof}
For every $3$-vertex graph $H$, interiority gives $\mathbbm{E}[\#H(W_n)]\geq\eta^3\binom n3$. The expectations satisfy both count identities in \cref{thm:main}. We choose nearby integers satisfying these identities while keeping the prescribed counts fixed.

Suppose first that the edge count is not prescribed. If three induced counts are prescribed, the fourth equals its expectation; integrality and divisibility follow from the assumed completion. With two prescribed counts, the two free counts have a fixed sum. Eliminating one gives a solvable congruence for the other modulo $n-2$, whose solutions have spacing at most $n-2$. Choose the solution nearest its expectation and recover the other count from the sum. Both errors are at most $(n-2)/2$.

With one prescribed induced count, choose two free counts whose graphs have consecutive numbers of edges; such a pair exists among any three of the four graphs. Round the remaining free count to the nearest integer. The resulting congruence has coefficient $1$ or $-1$, so it remains solvable. Choose the nearest solution for one of the two remaining counts and recover the other from their required sum. Their errors are at most $(n-2)/2$ and $(n-2)/2+1/2$.

Now suppose the edge count is prescribed. If two induced counts are also prescribed, the sum and weighted sum force the other two to equal their expectations; the assumed completion makes them integers. Otherwise, choose two free counts with consecutive edge numbers and round the expectations of the other free counts to the nearest integers. The two identities determine the chosen pair through a matrix of determinant $1$ or $-1$, so the resulting counts are integers. Since the edge numbers are in $\{0,1,2,3\}$, each of these two errors is at most $5/2$.

All four induced-count errors are at most $2n$, so $\eta^3\binom n3\geq4n$ ensures nonnegativity. The weighted count identity then gives a nonnegative integer edge target, and for $n\geq4$,
\[
\bigl|N_{\cedge[0.15]{black}}'-\mathbbm{E}[\#\cedge[0.15]{black}(W_n)]\bigr|
\leq\frac{(1+2+3)\,2n}{n-2}
=\frac{12n}{n-2}\leq24.
\]
This error is zero when the edge count is prescribed.
\end{proof}

Together, \cref{lem:s3-representatives,lem:s3-arithmetic} provide a sampling graphon and nearby integer targets. We now modify a sample in stages at decreasing scales. The first stage uses the edge response to reduce the errors in all induced $3$-vertex counts to $O_\delta(n)$.

\section{Correcting an ordinary sample}\label{sec:first}
For the rest of the proof, fix $\delta,n,\mathcal{H}$ as in \cref{thm:main}. Let $W$ and $(N_{H})_{H\in\{\cedge[0.1]{black},\ctriangle[0.2]{white}{white}{white},\ctriangle[0.2]{black}{white}{white},\ctriangle[0.2]{white}{black}{black},\ctriangle[0.2]{black}{black}{black}\}}$ be the graphon and target counts supplied by \cref{lem:s3-representatives,lem:s3-arithmetic}, respectively. We take $\eta_\delta=\eta/2=\Theta(\delta)$, with $\eta$ as in \cref{lem:s3-representatives}, as the interiority margin in \cref{lem:s3-arithmetic} and in the modifications below. Fix a partition
\[
 [n]=A\sqcup B\sqcup C\sqcup D
\]
into four sets with sizes differing by at most one, with $|A|=\lfloor n/4\rfloor$. Each correction stage uses edges with both endpoints in its assigned set. Since no $3$-vertex set contains two disjoint edges, modifications in one set leave the increments of modifications in the other sets unchanged.

We track all induced $3$-vertex counts simultaneously. By \cref{sec:prelim}, this is equivalent to tracking the edge, $2$-star and triangle densities, which we use for the first-order correction.

\begin{lemma}[First-order correction]\label{lem:first}
For sufficiently large $n$, depending only on $\delta$, with probability at least $9/10$ the sample $W_n$ admits additions and removals of edges inside $A$ after which every induced $3$-vertex count differs from its expectation in $W_n$ by $O_\delta(n)$.
\end{lemma}

\begin{proof}
By \cref{lem:orthogonal-decomposition}, each induced $3$-vertex count in $W_n$ has variance $O(n^5)$. Chebyshev's inequality and a union bound give simultaneous $O(n^{5/2})$ deviations from expectation with probability arbitrarily close to one. We keep the vertex variables fixed and change only the pair variables $Y_{jk}$. We use the edge response $D_W$ to establish a supply of useful actual increments in every direction, then show deterministically that this supply persists during the modifications.

Let $W_{n,i}$ be the graph after $i$ modifications, with $W_{n,0}=W_n$. We scale the density errors so that a single toggle changes the edge coordinate by $\pm1$, defining the residual
\[
r_i(W_n)=\binom{n}{2}
\bigl(\mathbbm{E}[t(H,W_n)]-t(H,W_{n,i})\bigr)_{
H\in\{\cedge[0.15]{black},\ctriangle[0.2]{white}{black}{black},\ctriangle[0.2]{black}{black}{black}\}}.
\]
The preceding fluctuation bounds and the count identities imply
$\norm{r_0(W_n)}\le Cn^{3/2}$ with probability at least $99/100$, by choosing an absolute constant $C$ sufficiently large.
For a candidate pair $jk\subseteq A$, let $E_i(X_j,X_k)$ denote the actual change in the density vector, multiplied by $\binom{n}{2}$, when that pair is toggled in $W_{n,i}$. Here the arguments identify the endpoints; this increment depends on the whole current graph. Counting the affected triples gives
\[
E_i(X_j,X_k)=\bigl(1-2\mathbbm{1}_{jk\in E(W_{n,i})}\bigr)
\begin{pmatrix}
1\\[2pt]
\displaystyle\frac{\sum_{z\in[n]\setminus\{j,k\}}
 (\mathbbm{1}_{jz\in E(W_{n,i})}+\mathbbm{1}_{kz\in E(W_{n,i})})}{n-2}\\[6pt]
\displaystyle\frac{3\sum_{z\in[n]\setminus\{j,k\}}
 \mathbbm{1}_{jz\in E(W_{n,i})}\mathbbm{1}_{kz\in E(W_{n,i})}}{n-2}
\end{pmatrix}.
\]
In particular, $\norm{E_i(X_j,X_k)}\le4$, and toggling this pair replaces $r_i(W_n)$ by $r_i(W_n)-E_i(X_j,X_k)$. In the original sample,
\[
\mathbbm{E}\!\left[E_0(X_j,X_k)\mid
 X_j,X_k,\mathbbm{1}_{jk\in E(W_n)}\right]
 =\bigl(1-2\mathbbm{1}_{jk\in E(W_n)}\bigr)D_W(X_j,X_k).
\]
Both possible states of $jk$ have conditional probability at least $\eta_\delta$. Thus conditional Jensen's inequality and \cref{lem:s3-transformation} give, for every unit vector $u\in\mathbb R^3$,
\[
\mathbbm{E}\bigl[\max\{u\cdot E_0(X_j,X_k),0\}\bigr]
\ge\eta_\delta\int_{[0,1]^2}|u\cdot D_W(x,y)|\,dx\,dy
\ge\frac{\eta_\delta^3}{9}.
\]
Set $c=\eta_\delta^3/36$. We seek a lower bound valid simultaneously in every direction. For a fixed unit vector $u$, apply bounded differences to the following average, whose expectation is at least $4c$:
\[
\binom{|A|}{2}^{-1}\sum_{\substack{j<k\\j,k\in A}}
 \max\{u\cdot E_0(X_j,X_k),0\}.
\]
Toggling one graph edge changes this average by at most $C/n^2$: its own increment changes by at most $8$, and only $O(n)$ other candidate increments can change, each by $O(1/n)$. Changing one vertex variable changes at most $n-1$ graph edges, so its bounded-difference constant is $C/n$; changing one pair variable has constant $C/n^2$. The sum of the squares of these constants is $O(1/n)$. McDiarmid's inequality therefore bounds the probability that the displayed average is below $3c$ by $2\exp(-c^2n/C)$. The average is $4$-Lipschitz in $u$. Applying this estimate to a $c/4$-net of the unit sphere and taking a union bound proves that, with probability $1-o_\delta(1)$, the average is at least $2c$ for every unit vector $u$.

On this event, proceed deterministically and use each candidate pair at most once. The same single-edge bound shows that after any $i$ adaptive toggles, and after omitting the used candidates,
\[
\inf_{\norm{u}=1}\binom{|A|}{2}^{-1}
\sum_{\substack{j<k,\ j,k\in A\\jk\text{ unused}}}
\max\{u\cdot E_i(X_j,X_k),0\}
\ge2c-\frac{Ci}{n^2}.
\]
Indeed, the changes to the average cost at most $Ci/n^2$, and omitting $i$ summands costs at most $4i/\binom{|A|}{2}$, which is absorbed by increasing $C$. For sufficiently large $n$, depending only on $\delta$, the right-hand side is at least $c$ throughout the first $\lceil4Cc^{-1}n^{3/2}\rceil$ steps. Consequently, whenever $\norm{r_i(W_n)}>16/c$, there is an unused pair $jk$ such that
\[
r_i(W_n)\cdot E_i(X_j,X_k)\ge c\norm{r_i(W_n)}.
\]
Toggle this pair. Then
\[
\begin{aligned}
\norm{r_{i+1}(W_n)}^2
&\le\norm{r_i(W_n)}^2-2c\norm{r_i(W_n)}+16\\
&\le\bigl(\norm{r_i(W_n)}-c/2\bigr)^2.
\end{aligned}
\]
Starting from $\norm{r_0(W_n)}\le Cn^{3/2}$, the process therefore stops within the available number of steps with $\norm{r_i(W_n)}\le16/c$. Each of the three density errors is then $O_\delta(n^{-2})$. Their fixed linear relations with the induced densities give an $O_\delta(n)$ error for every induced $3$-vertex count. The fluctuation and directional-supply events intersect with probability at least $9/10$ for sufficiently large $n$, proving the claim.
\end{proof}

\section{Paired corrections}\label{sec:paired}

After the correction in \cref{lem:first}, every induced $3$-vertex count is within $O_\delta(n)$ of its expectation in the original sample $W_n$. Comparing the identity
$(n-2)\#\cedge[0.1]{black}
=\#\ctriangle[0.2]{black}{white}{white}
+2\#\ctriangle[0.2]{white}{black}{black}
+3\#\ctriangle[0.2]{black}{black}{black}$
with its expectation shows that the edge count is within $O_\delta(1)$ of its expectation. By \cref{lem:s3-arithmetic}, the integer target $N_{\cedge[0.1]{black}}$ differs from that expectation by at most $24$. Thus only $O_\delta(1)$ edge additions or removals are needed. We make this correction inside $B$ before proceeding with modifications that preserve the edge count.

\begin{lemma}[Fixing the edge count]\label{lem:edge-correction}
Fix $C>0$. For sufficiently large $n$, depending only on $\delta,C$, with probability $1-o(1)$ the following holds simultaneously for every integer $0\leq m\leq C$: the edge count of $W_n$ can be increased or decreased by $m$ using $m$ additions or removals inside $B$, respectively, with every induced $3$-vertex count changing by at most $m(n-2)$.
\end{lemma}

\begin{proof}
Conditional on the vertex variables, the edge indicators inside $B$ are independent, with probabilities in $(\eta_\delta,1-\eta_\delta)$. Since $|B|=\Theta(n)$, a Chernoff bound gives at least $C$ edges and at least $C$ non-edges inside $B$ with probability $1-o(1)$, uniformly in the vertex variables. Adding or removing any $m$ of them changes each induced $3$-vertex count by at most $m(n-2)$, since each altered edge lies in exactly $n-2$ triples. The same event works for every $m$ and both signs.
\end{proof}

The modifications inside $A$ leave all edge states inside $B$ unchanged. Applying \cref{lem:edge-correction} with $C$ depending only on $\delta$ therefore makes the edge count exactly $N_{\cedge[0.1]{black}}$, while leaving every induced $3$-vertex count within $O_\delta(n)$ of its expectation in $W_n$. Since the corresponding targets differ from these expectations by at most $2n$, the remaining errors relative to the targets are also $O_\delta(n)$.

The count identities in \cref{thm:main} imply that, once the edge count is fixed, it suffices to correct the counts of $\ctriangle[0.2]{white}{white}{white}$ and $\ctriangle[0.2]{black}{black}{black}$. We do this using swaps, each removing one edge and adding another, so the edge count remains fixed.

From this point until the end of the precise-swap argument, we condition on all vertex variables $(X_i)_{i\in[n]}$. The edge indicators are then independent, with probabilities $W(X_i,X_j)\in(\eta_\delta,1-\eta_\delta)$, and all probability estimates below are uniform in the conditioned values.

\subsection{Local probabilities for a swap}

We need lower bounds for probabilities of prescribed swap increments near their mean, and upper bounds at every lattice point. The contributions from outside vertices are independent but generally not identically distributed, so the following local law allows their distributions to vary.

\begin{lemma}[Two-dimensional local law]\label{lem:local}
Let $\xi_1,\ldots,\xi_\ell$ be independent vectors in $\{-1,0,1\}^2$ such that, for every $j$,
\[
P(\xi_j=(0,0)),\quad P(\xi_j=(1,0)),\quad P(\xi_j=(0,1))\geq\gamma,
\qquad 0<\gamma\leq1/3.
\]
Put $Q=\sum_j\xi_j$. For an absolute constant $C$,
\[
\sup_{z\in\mathbb Z^2}P(Q=z)\leq\frac{C\gamma^{-2}}{\ell}.
\]
For fixed $R\geq1$ and $\ell\geq Ce^{R^2}\gamma^{-10}$, every $z\in\mathbb Z^2$ satisfies
\begin{equation}\label{eq:local-lower}
\lVert z-\mathbbm E Q\rVert\leq R\gamma\sqrt\ell
\quad\Longrightarrow\quad
P(Q=z)\geq\frac{c_R}{\ell},
\qquad c_R=\frac{e^{-R^2/2}}{8\pi}.
\end{equation}
\end{lemma}

\begin{proof}
Put $\Sigma=\Cov(Q)$ and
$\phi_j(t)=\mathbbm E e^{\mathrm{i}t\cdot(\xi_j-\mathbbm E\xi_j)}$.
The three specified atoms give
\[
\gamma^2 I\preceq\Cov(\xi_j)\preceq2I,
\qquad
|\phi_j(t)|\leq e^{-c\gamma^2\lVert t\rVert^2}
\quad(t\in[-\pi,\pi]^2),
\]
where $c>0$ is absolute. Fourier inversion gives the point-mass upper bound.

Boundedness and Taylor's formula give
$|\phi_j(t)-e^{-t^T\Cov(\xi_j)t/2}|\leq C\lVert t\rVert^3$
on the same square. Both factors have modulus at most
$e^{-c\gamma^2\lVert t\rVert^2}$.
Telescoping their products and applying Fourier inversion, including the Gaussian tail outside the square, gives, for $\ell\geq2$,
\begin{equation}\label{eq:local-approx}
\sup_{z\in\mathbb Z^2}\left|P(Q=z)-
\frac{e^{-\frac12(z-\mathbbm E Q)^T\Sigma^{-1}(z-\mathbbm E Q)}}
{2\pi\sqrt{\det\Sigma}}\right|
\leq C\gamma^{-5}\ell^{-3/2}.
\end{equation}
Finally, $\gamma^2\ell I\preceq\Sigma\preceq2\ell I$.
If $\lVert z-\mathbbm E Q\rVert\leq R\gamma\sqrt\ell$, the Gaussian term is at least $e^{-R^2/2}/(4\pi\ell)$.
For $\ell\geq Ce^{R^2}\gamma^{-10}$, the error is at most half this quantity, proving \eqref{eq:local-lower}.
\end{proof}

We apply \cref{lem:local} to the change caused by a swap.

\begin{lemma}[Effect of a swap]\label{lem:swap-change}
Let $ij\neq lk$ be two pairs of vertices in $W_n$, not necessarily disjoint. To define $D_{ij,lk}(W_n)$, first force $ij$ to be present and $lk$ to be absent, leaving all other edges as in $W_n$, and then measure the change to $(\#\ctriangle[0.2]{white}{white}{white},\#\ctriangle[0.2]{black}{black}{black})$ when removing $ij$ and adding $lk$. Thus $D_{ij,lk}(W_n)$ is defined for every sample, is independent of the two candidate edge indicators conditional on the vertex variables, and equals the actual swap increment whenever $ij\in E(W_n)$ and $lk\notin E(W_n)$. For any $z\in\mathbb Z^2$ satisfying $\lVert z-\mathbbm E[D_{ij,lk}(W_n)]\rVert=O(\sqrt n)$, we have
\[
 P(D_{ij,lk}(W_n)=z)=\Theta_\delta(n^{-1}),
\]
uniformly in the vertex variables, the candidate pair, and $z$ in a fixed such range. Moreover, the upper bound $O_\delta(n^{-1})$ holds for every $z\in\mathbb Z^2$. For every fixed $R\ge1$, the lower bound can be taken to be $c_{2R+2}/n$ whenever
\[
 \lVert z-\mathbbm E[D_{ij,lk}(W_n)]\rVert\le R\eta_\delta^4\sqrt n,
\]
for sufficiently large $n$ depending on $\delta,R$.
\end{lemma}
\begin{proof}
First suppose that $ij,lk$ are disjoint. Each vertex $m\notin\{i,j,l,k\}$ contributes
\[
\bigl(
\mathbbm{1}_{im,jm\notin E(W_n)}-\mathbbm{1}_{lm,km\notin E(W_n)},
\mathbbm{1}_{lm,km\in E(W_n)}-\mathbbm{1}_{im,jm\in E(W_n)}
\bigr)
\]
to $D_{ij,lk}(W_n)$. These contributions are independent over $m$ and independent of the bounded total contribution of the four triples contained in $\{i,j,l,k\}$. Each outside contribution lies in $\{-1,0,1\}^2$ and assigns probability at least $\eta_\delta^4$ to each of $(0,0),(1,0),(0,1)$: each value is obtained by prescribing a suitable configuration of the four relevant edge indicators. None uses either candidate edge indicator.

If the two edges share a vertex, write them as $ij,ik$. Their common triple keeps the same number of edges and contributes zero. For an outside vertex $m$, its contribution is
\[
\bigl(
(1-\mathbbm1_{im\in E(W_n)})(\mathbbm1_{km\in E(W_n)}-\mathbbm1_{jm\in E(W_n)}),
\mathbbm1_{im\in E(W_n)}(\mathbbm1_{km\in E(W_n)}-\mathbbm1_{jm\in E(W_n)})
\bigr).
\]
These contributions are again independent over $m$. Prescribing $jm$ absent and $km$ present gives the positive unit atoms according to the state of $im$, each with probability at least $\eta_\delta^3$; equal states of $jm,km$ give the zero atom. Thus the lower atom bound $\eta_\delta^4$ also holds in this case.

Condition additionally on the edges inside the candidate support and apply \cref{lem:local}, with $\gamma=\eta_\delta^4$, to the $n-O(1)$ outside summands. The internal contribution shifts the relevant center by only $O(1)$, so averaging over these internal edges gives the asserted bounds. For the quantitative lower bound, this shift and the change from $n$ to $n-O(1)$ are absorbed by using radius parameter $2R+2$ for sufficiently large $n$.
\end{proof}

\subsection{A supply of useful pairs}

The lower bound in \cref{lem:swap-change} applies to increments near their mean. To obtain useful increments near zero, we choose many candidate pairs whose mean increments are $O(1)$.

\begin{lemma}[Pairs with small mean]\label{lem:many-pairs}
Let $V\subset[n]$ with $|V|=\Theta(n)$. Using only the vertex variables, we can choose $\Theta(n^2)$ ordered pairs $(e,e')$ of distinct possible edges in $V$, with each edge appearing in at most one pair, such that
\[
 \mathbbm E[D_{e,e'}(W_n)]=O(1)
\]
uniformly over the pairs. With probability $1-o(1)$, $\Theta_\delta(n^2)$ of these pairs also satisfy $e\in E(W_n)$ and $e'\notin E(W_n)$.
\end{lemma}
\begin{proof}
For each possible edge $ij\subset V$, consider the vector
\[
\left(
\sum_{m\in[n]\setminus\{i,j\}}(1-W(X_i,X_m))(1-W(X_j,X_m)),
\sum_{m\in[n]\setminus\{i,j\}}W(X_i,X_m)W(X_j,X_m)
\right)\in[0,n]^2.
\]
Its coordinates are the expected numbers of empty triples destroyed and triangles created when the edge is added. Partition $[0,n]^2$ into at most $cn^2$ squares of bounded side length, for a sufficiently small fixed $c>0$ depending only on the lower bound for $|V|/n$. Within each square, pair consecutive edges in a fixed order, taking the earlier edge first. At most one edge per square remains unpaired, so $\Theta(n^2)$ ordered pairs are obtained and no edge is reused.

For disjoint paired edges, the expected swap increment is the difference of their first coordinates and the reverse difference of their second coordinates. For incident paired edges, omitting the common triple changes this comparison by only $O(1)$. Consequently every chosen pair has $\mathbbm E[D_{e,e'}(W_n)]=O(1)$.

Each event $e\in E(W_n),e'\notin E(W_n)$ has probability at least $\eta_\delta^2$. These events are independent across the chosen pairs because no edge is reused. A Chernoff bound therefore gives $\Theta_\delta(n^2)$ pairs in the prescribed state with probability $1-o(1)$. For each pair, its prescribed state is independent of its increment by \cref{lem:swap-change}.
\end{proof}

\section{Coarse Swaps}\label{sec:coarse}

After the first-order and edge-count corrections, every induced $3$-vertex count is within $O_\delta(n)$ of its target. We now reduce the errors in the empty-triple and triangle counts to at most $2\eta_\delta^4\sqrt n$, while preserving the edge count. Only $O_\delta(\sqrt n)$ swaps with increments of order $\sqrt n$ are needed. The supply from \cref{lem:many-pairs,lem:swap-change} will let us choose these swaps with disjoint vertex supports, so their increments remain unchanged throughout the correction.

\begin{lemma}[Coarse swaps]\label{lem:many-swaps}
For any fixed implied bound in $|a|+|b|=O_\delta(n)$, with probability $1-o(1)$ the following holds simultaneously for every integer pair $a,b$ in that range. There are $O_\delta(\sqrt n)$ swaps with pairwise disjoint vertex supports inside $C$ such that the resulting graph $W_n'$ satisfies
\[
\begin{gathered}
|\#\ctriangle[0.3]{white}{white}{white}(W_n')-\#\ctriangle[0.3]{white}{white}{white}(W_n)-a|\le 2\eta_\delta^4\sqrt n,\\
|\#\ctriangle[0.3]{black}{black}{black}(W_n')-\#\ctriangle[0.3]{black}{black}{black}(W_n)-b|\le 2\eta_\delta^4\sqrt n,\qquad
\#\cedge[0.2]{black}(W_n')=\#\cedge[0.2]{black}(W_n).
\end{gathered}
\]
\end{lemma}
\begin{proof}
Apply \cref{lem:many-pairs} inside $C$, obtaining candidate pairs $(e_i,e_i')_{i\in[m]}$, where $m=\Theta(n^2)$, chosen using only the vertex variables. Their expected increments are uniformly $O(1)$. We count candidates in the prescribed state with increments in each sign quadrant. For $x<y$ and $z<w$, put
\[
W_{x,y,z,w}(W_n)=\sum_{i\in[m]}\mathbbm1\left[
\begin{gathered}
e_i\in E(W_n),\quad e_i'\notin E(W_n),\\
D_{e_i,e_i'}(W_n)\in[x\eta_\delta^4\sqrt n,y\eta_\delta^4\sqrt n]\times[z\eta_\delta^4\sqrt n,w\eta_\delta^4\sqrt n]
\end{gathered}\right].
\]
Use the four choices
\[
(x,y,z,w)\in\{(1,2,1,2),\ (-2,-1,1,2),\ (1,2,-2,-1),\ (-2,-1,-2,-1)\}.
\]
For sufficiently large $n$, each rectangle contains $\Theta(\eta_\delta^8 n)$ lattice points within a fixed multiple of $\eta_\delta^4\sqrt n$ of each candidate's mean. By \cref{lem:swap-change}, the probability of each rectangle is $\Omega(\eta_\delta^8)$. The prescribed state is independent of the increment and has probability at least $\eta_\delta^2$. Thus each candidate is useful with probability $\Omega(\eta_\delta^{10})$, and $\mathbbm E[W_{x,y,z,w}(W_n)]=\Theta_\delta(n^2)$.

Two pairs with disjoint vertex supports need not have independent increments, since they share the at most sixteen edges joining their supports. Condition on those edges. All remaining indicators used by one pair are distinct from those used by the other, so their usefulness indicators are conditionally independent. Changing the exposed edges moves either rectangle event by only a bounded amount. After also fixing the internal edges of a support, the remaining increment is a sum of $n-O(1)$ independent terms as in \cref{lem:swap-change}. A bounded movement of the rectangle boundary crosses $O(\sqrt n)$ lattice points, each having probability $O_\delta(n^{-1})$. Thus the conditional probability of usefulness varies by $O_\delta(n^{-1/2})$ as the exposed edges vary. The covariance of the two usefulness indicators is consequently $O_\delta(n^{-1})$.

Each vertex belongs to at most $n$ candidate pairs, since candidate edges are not reused. There are therefore $O(n^3)$ ordered pairs of candidates whose supports intersect. Bounding their covariances by one and using the preceding bound for disjoint supports gives $
\operatorname{Var}(W_{x,y,z,w}(W_n))=O_\delta(n^3)$.

Chebyshev's inequality and a union bound over the four choices show that, with probability $1-o(1)$, each of the four classes contains $\Theta_\delta(n^2)$ usable pairs.

On this event, choose swaps with pairwise disjoint vertex supports. Each choice excludes at most $4n$ candidates, since its support has at most four vertices and each vertex belongs to at most $n$ candidate pairs. Thus $O_\delta(\sqrt n)$ choices exclude only $O_\delta(n^{3/2})=o(n^2)$ candidates and cannot exhaust any of the four classes for sufficiently large $n$. Every remaining candidate retains its original edge states and increment: no triple contains altered edges from two disjoint supports.

Starting with residuals $a,b$, choose at each step the sign quadrant pointing toward both current residuals, with either sign allowed for a zero coordinate. Each coordinate increment has absolute value between $\eta_\delta^4\sqrt n$ and $2\eta_\delta^4\sqrt n$. Once a residual coordinate enters $[-2\eta_\delta^4\sqrt n,2\eta_\delta^4\sqrt n]$, it stays there; while outside, its absolute value decreases by at least $\eta_\delta^4\sqrt n$ per step. Consequently at most
\[
2+\frac{|a|+|b|}{\eta_\delta^4\sqrt n}=O_\delta(\sqrt n)
\]
swaps suffice. Each swap preserves the edge count. On the same supply event, this deterministic argument works for every $a,b$ in the stated range. Changes inside $A$ or $B$ alter neither the candidate states nor their increments, by the same disjoint-support observation, so the assertion also applies to residuals determined by the preceding corrections.
\end{proof}

\section{Precise swaps}\label{sec:precise}

After the corrections inside $A,B,C$ from \cref{lem:first,lem:edge-correction,lem:many-swaps}, the edge count is exactly its target, and each of the counts $(\#\ctriangle[0.2]{white}{white}{white},\#\ctriangle[0.2]{black}{black}{black})$ is within $2\eta_\delta^4\sqrt n$ of its target. We now seek two swaps inside $D$ whose increments sum to the remaining correction. Since this correction depends on the preceding stages, such swaps must be available simultaneously for every integer target in this range.

We choose two families of $\Theta(n^{0.6})$ candidate swaps, with disjoint vertex supports across both families. Exposing the edges among all their endpoints makes the remaining increments conditionally independent; a concentration estimate keeps their conditional means near zero. The first family supplies many distinct increments, and the second supplies an increment completing one of them to the required target.

\begin{lemma}[Conditional independence]\label{lem:approx-independence}
Before exposing the edge variables, one can choose pairs $(e_i,e_i')_{i\in[m]}$ contained in $D$, where $m=2\lfloor n^{0.6}\rfloor$, such that
\begin{itemize}
    \item any two pairs are disjoint, not sharing a vertex;
    \item $\mathbbm{E}[D_{e_i,e_i'}(W_n)]=O(1)$ uniformly over all pairs.
\end{itemize}
Expose all edge variables with both endpoints in $\bigcup_{i\in[m]}(e_i\cup e_i')$. With probability
\[
1-O\bigl(e^{-c\eta_\delta^2n^{0.6}}+n^{0.6}e^{-cn^{0.2}}\bigr),
\]
each half of the family contains $\Omega(\eta_\delta^2n^{0.6})$ pairs with $e_i\in E(W_n)$ and $e_i'\notin E(W_n)$. On the same event, conditional on the exposed edges, their changes are mutually independent, and for every $z\in\mathbb Z^2\cap[-4\eta_\delta^4\sqrt n,4\eta_\delta^4\sqrt n]^2$,
\[
\frac{c}{n}\leq
P\bigl(D_{e_i,e_i'}(W_n)=z\mid\text{the exposed edges}\bigr)
\leq\frac{C\eta_\delta^{-8}}{n}.
\]
The upper bound holds for every $z\in\mathbb Z^2$.
\end{lemma}
\begin{proof}
    By \cref{lem:many-pairs}, there exist $\Theta(n^2)$ candidate pairs with expected change $O(1)$ contained in $D$, chosen before exposing the edge variables. Since no edge appears in more than one pair, each vertex belongs to at most $n-1$ pairs. Choosing one pair and discarding every pair sharing one of its endpoints therefore removes at most $4n$ candidates. As $m=o(n)$, this gives the required $m$ disjoint pairs. The events $e_i\in E(W_n),e_i'\notin E(W_n)$ are independent and each has probability at least $\eta_\delta^2$. A Chernoff bound gives the asserted number in each half, except with probability $O(e^{-c\eta_\delta^2n^{0.6}})$.

    To retain the central local lower bound after exposure, we control the change in conditional means. For each pair, contributions from third vertices in the union of the endpoints consist of $O(n^{0.6})$ independent bounded summands, together with a bounded contribution from triples entirely within its own endpoints. Exposing the edges within the union changes the conditional mean by the centered total of these contributions. Hoeffding's inequality and a union bound over all $m$ pairs therefore give
    \[
    \max_{i\in[m]}\left\|
    \mathbbm{E}[D_{e_i,e_i'}(W_n)\mid\text{the exposed edges}]
    -\mathbbm{E}[D_{e_i,e_i'}(W_n)]
    \right\|_\infty\leq n^{0.4}
    \]
    except with probability $O(n^{0.6}e^{-cn^{0.2}})$. On this event every conditional mean is $O(1)+O(n^{0.4})$, so both coordinates are bounded in absolute value by $\eta_\delta^4\sqrt n$ for sufficiently large $n$.

    Conditional on the exposed edges, the remaining contributions use edges joining each pair's endpoints to vertices outside the union. These sets of edge variables are disjoint for different pairs, so the changes are mutually independent. Each remaining sum has $n-O(n^{0.6})$ independent summands satisfying \cref{lem:local} with $\gamma=\eta_\delta^4$. The conditional-mean bound and \cref{lem:local}, with $R=10$, give the stated lower bound throughout the box and the upper bound everywhere. Intersecting the two events proves the result.
\end{proof}

We now use this conditional independence to obtain all required precise corrections simultaneously.

\begin{lemma}[Precise modifications]\label{lem:precise-swaps}
With probability $1-o(1)$, simultaneously for every $(a,b)\in\mathbb Z^2\cap[-2\eta_\delta^4\sqrt n,2\eta_\delta^4\sqrt n]^2$, two swaps inside $D$ change $(\#\ctriangle[0.2]{white}{white}{white},\#\ctriangle[0.2]{black}{black}{black})$ in $W_n$ by precisely $(a,b)$.
\end{lemma}
\begin{proof}
     Fix $(a,b)\in\mathbb Z^2\cap[-2\eta_\delta^4\sqrt n,2\eta_\delta^4\sqrt n]^2$. Expose the edges specified in \cref{lem:approx-independence} and condition on a realization satisfying its conclusion. Retain the first-half pairs with $e_i\in E(W_n),e_i'\notin E(W_n)$. Each has probability $\Omega(\eta_\delta^8)$ of inducing a change in $[-2\eta_\delta^4\sqrt n,2\eta_\delta^4\sqrt n]^2$, by summing the local lower bound over the lattice points in this box. Conditional independence and a Chernoff bound therefore give $\Omega(\eta_\delta^{10}n^{0.6})$ such pairs, except with probability $O(e^{-c\eta_\delta^{10}n^{0.6}})$.

     The local upper bound also controls multiplicities. For any ten first-half pairs and any fixed lattice value, the probability that all ten induce that value is at most $(C\eta_\delta^{-8}/n)^{10}$. A union bound over these sets and the lattice points in the box gives probability $O(\eta_\delta^{-80}n^{-3})$ that any value occurs ten times. Otherwise each value occurs at most nine times, so the first half supplies $\Omega(\eta_\delta^{10}n^{0.6})$ distinct changes in the box.

    Expose all first-half changes and fix a realization with this property. The increments completing the distinct first-half changes to $(a,b)$ are distinct lattice points in $[-4\eta_\delta^4\sqrt n,4\eta_\delta^4\sqrt n]^2$. By conditional independence, the second-half increments retain their distributions and remain independent after this exposure. Summing the local lower bound over the required complements shows that each second-half pair with $e_i\in E(W_n),e_i'\notin E(W_n)$ succeeds with probability $\Omega(\eta_\delta^{10}n^{-0.4})$. Since there are $\Omega(\eta_\delta^2n^{0.6})$ such pairs, their joint failure probability is at most
    \[
    \bigl(1-c\eta_\delta^{10}n^{-0.4}\bigr)^{c\eta_\delta^2n^{0.6}}
    \leq e^{-c\eta_\delta^{12}n^{0.2}},
    \]
    after adjusting the absolute constant $c$.
    The successful first-half event is common to all targets. Union bounding over the $O(n)$ integer values of $(a,b)$, and including the preceding exceptional events, gives total failure probability
    \[
    O\bigl(\eta_\delta^{-80}n^{-3}+e^{-c\eta_\delta^{10}n^{0.6}}
    +n^{0.6}e^{-cn^{0.2}}+ne^{-c\eta_\delta^{12}n^{0.2}}\bigr)=o(1).
    \]
    Finally, the two selected pairs have disjoint endpoints, so no vertex triplet contains an altered edge from both swaps. Their changes therefore add exactly, proving the result.
\end{proof}

\section{Proof of the main theorem}\label{sec:proof}
\begin{proof}[Proof of Theorem~\ref{thm:main}]
Necessity follows from the identities between induced counts. For sufficiency, fix $\delta,\mathcal{H},v$ satisfying the hypotheses and suppose an arithmetic completion exists. Lemma~\ref{lem:s3-representatives} gives a graphon $W=F_\eta U$ with $\mathbbm{E}[\#H(W_{n})]=\binom{n}{v(H)}v_{H}$ for all $H\in\mathcal{H}$, satisfying the assumptions of \cref{lem:first,lem:edge-correction,lem:many-swaps,lem:precise-swaps}. For sufficiently large $n$, depending only on $\delta$, \cref{lem:s3-arithmetic} supplies admissible integer targets agreeing with every prescribed count, within $2n$ of the expected $3$-vertex counts and within $24$ of the expected edge count.

Sample $W_n$ with the fixed partition into $A,B,C,D$. The event in \cref{lem:first} has probability at least $9/10$, and those in \cref{lem:edge-correction,lem:many-swaps,lem:precise-swaps} each have probability $1-o(1)$. The latter estimates are uniform in the vertex variables, so a union bound gives a common event of probability at least $9/10-o(1)>0$. Work on this event. The later corrections are available simultaneously for all required residuals, including those determined by earlier stages. Modifications inside different parts leave one another's increments unchanged, since no triple contains edges lying inside two different parts.

First apply \cref{lem:first} inside $A$, obtaining $W_n'$ with every induced $3$-vertex count within $O_\delta(n)$ of its expectation in $W_n$ and hence, by \cref{lem:s3-arithmetic}, within $O_\delta(n)$ of its integer target. Comparing the weighted count identity for $W_n'$ with the target identity shows that its edge count is within $O_\delta(1)$ of its target.

Next apply \cref{lem:edge-correction} inside $B$, changing $O_\delta(1)$ edges to obtain $W_n''$ with exactly the target edge count. Each induced $3$-vertex count changes by $O_\delta(n)$ and is therefore within $O_\delta(n)$ of its integer target.

Apply \cref{lem:many-swaps} inside $C$, using $O_\delta(\sqrt n)$ swaps to obtain $W_n'''$. The edge count remains fixed, and each of $(\#\ctriangle[0.2]{white}{white}{white}(W_n'''),\#\ctriangle[0.2]{black}{black}{black}(W_n'''))$ is within $2\eta_\delta^4\sqrt n$ of its target.

Finally, \cref{lem:precise-swaps} supplies two swaps inside $D$ that remove the remaining residual, producing $W_n''''$ with the same edge count and with $(\#\ctriangle[0.2]{white}{white}{white}(W_n''''),\#\ctriangle[0.2]{black}{black}{black}(W_n''''))$ exactly at their targets.

The resulting graph and the chosen targets both satisfy the identities
\[
(n-2)\#\cedge[0.2]{black}=\#\ctriangle[0.3]{black}{white}{white}+2\#\ctriangle[0.3]{white}{black}{black}+3\#\ctriangle[0.3]{black}{black}{black},\qquad \#\ctriangle[0.3]{white}{white}{white}+\#\ctriangle[0.3]{black}{white}{white}+\#\ctriangle[0.3]{white}{black}{black}+\#\ctriangle[0.3]{black}{black}{black}=\binom{n}{3}.
\]
Since the edge, empty-triple and triangle counts agree with their targets, these identities force the remaining two induced counts to agree as well. Thus every prescribed count is realized.

The thresholds in the arithmetic and first-order corrections are polynomial in $\eta_\delta^{-1}$. By \cref{lem:local}, the swap estimates on the scale $\eta_\delta^4\sqrt n$ also have polynomial thresholds and probability constants bounded by powers of $\eta_\delta$. The finite bounds used in the coarse and precise corrections therefore require only a polynomial lower bound on $n$ in terms of $\eta_\delta^{-1}$. Since $\eta_\delta=\Theta(\delta)$, the threshold in the theorem may be chosen polynomial in $\delta^{-1}$.
\end{proof}

\section*{Acknowledgments}

The author is grateful to Chaim Even-Zohar for his guidance on earlier research that inspired the present work. Thanks are also due to Michael Krivelevich for several helpful discussions and suggestions regarding the results and their presentation, particularly the emphasis on the edge–triangle case; to Mehtaab Sawhney for suggesting the graphon perspective; and to Svante Janson for pointing out errors in an earlier draft.

\bibliographystyle{alpha}
\bibliography{Exact-Realization}

@article{janson1992probabilistic,
  title={Probabilistic construction of proportional graphs},
  author={Janson, Svante and Spencer, Joel},
  journal={Random Structures \& Algorithms},
  volume={3},
  number={2},
  pages={127--137},
  year={1992},
  publisher={Wiley Online Library}
}

@article{huang20143,
  title={On the 3-local profiles of graphs},
  author={Huang, Hao and Linial, Nati and Naves, Humberto and Peled, Yuval and Sudakov, Benny},
  journal={Journal of Graph Theory},
  volume={76},
  number={3},
  pages={236--248},
  year={2014},
  publisher={Wiley Online Library}
}

@article{sah2024local,
  title={Local limit theorem for joint subgraph counts},
  author={Sah, Ashwin and Sawhney, Mehtaab and Zhu, Daniel G},
  journal={arXiv preprint arXiv:2412.09535},
  year={2024}
}

@article{lovasz2006limits,
  title={Limits of dense graph sequences},
  author={Lov{\'a}sz, L{\'a}szl{\'o} and Szegedy, Bal{\'a}zs},
  journal={Journal of Combinatorial Theory, Series B},
  volume={96},
  number={6},
  pages={933--957},
  year={2006},
  publisher={Elsevier}
}

@article{glebov2016densities,
  title={Densities of 3-vertex graphs},
  author={Glebov, Roman and Grzesik, Andrzej and Hu, Ping and Hubai, Tamas and Kral, Daniel and Volec, Jan},
  journal={arXiv preprint arXiv:1610.02446},
  year={2016}
}

@article{janson1991asymptotic,
  title={The asymptotic distributions of generalized U-statistics with applications to random graphs},
  author={Janson, Svante and Nowicki, Krzysztof},
  journal={Probability theory and related fields},
  volume={90},
  number={3},
  pages={341--375},
  year={1991},
  publisher={Springer}
}

@article{zhang2022berry,
  title={Berry--Esseen bounds for generalized U-statistics},
  author={Zhang, Zhuo-Song},
  journal={Electronic Journal of Probability},
  volume={27},
  pages={1--36},
  year={2022},
  publisher={The Institute of Mathematical Statistics and the Bernoulli Society}
}

@article{kittipassorn2025many,
  title={How many triangles can a graph have?},
  author={Kittipassorn, Teeradej and Popielarz, Kamil},
  journal={Bulletin of the Australian Mathematical Society},
  volume={111},
  number={3},
  pages={405--415},
  year={2025},
  publisher={Cambridge University Press}
}

@book{lovasz2012large,
  title={Large networks and graph limits},
  author={Lov{\'a}sz, L{\'a}szl{\'o}},
  volume={60},
  year={2012},
  publisher={American Mathematical Soc.}
}

@article{radin2013phase,
  title={Phase transitions in a complex network},
  author={Radin, Charles and Sadun, Lorenzo},
  journal={Journal of Physics A: Mathematical and Theoretical},
  volume={46},
  number={30},
  pages={305002},
  year={2013},
  publisher={IOP Publishing}
}

@article{bhattacharya2023fluctuations,
  title={Fluctuations of subgraph counts in graphon based random graphs},
  author={Bhattacharya, Bhaswar B and Chatterjee, Anirban and Janson, Svante},
  journal={Combinatorics, Probability and Computing},
  volume={32},
  number={3},
  pages={428--464},
  year={2023},
  publisher={Cambridge University Press}
}

@article{razborov2008minimal,
  title={On the minimal density of triangles in graphs},
  author={Razborov, Alexander A},
  journal={Combinatorics, Probability and Computing},
  volume={17},
  number={4},
  pages={603--618},
  year={2008},
  publisher={Cambridge University Press}
}

\end{document}